\documentclass[a4paper, 12pt]{amsart}

\usepackage{amssymb,amsmath}
\usepackage[dvips]{graphics}
\usepackage[dvips]{color}
\usepackage{graphicx}
\usepackage[english]{babel}
\usepackage[latin1]{inputenc}

\newtheorem{theorem}{Theorem}[section]
\newtheorem{lemma}[theorem]{Lemma}

\theoremstyle{definition}

\newtheorem{remark}[theorem]{Remark}

\numberwithin{equation}{section}

\title[The fractional maximal operator on Besov and T--L spaces]
{Smoothing properties of the discrete fractional maximal operator on 
Besov and Triebel--Lizorkin spaces 
}

\author{Toni Heikkinen}
\author{Heli Tuominen}

\newcommand\rn{\mathbb R^n}

\newcommand\n{\mathbb N}

\newcommand\ph{\varphi}
\newcommand\eps{\varepsilon}
\newcommand\M{\operatorname{\mathcal M}}

\providecommand{\ch}[1]{\text{\raise 2pt \hbox{$\chi$}\kern-0.2pt}_{#1}}

\providecommand{\vint}[1]{\mathchoice
          {\mathop{\vrule width 5pt height 3 pt depth -2.5pt
                  \kern -9pt \kern 1pt\intop}\nolimits_{\kern -5pt{#1}}}%
          {\mathop{\vrule width 5pt height 3 pt depth -2.6pt
                  \kern -6pt \intop}\nolimits_{\kern -3pt{#1}}}%
          {\mathop{\vrule width 5pt height 3 pt depth -2.6pt
                  \kern -6pt \intop}\nolimits_{\kern -3pt{#1}}}%
          {\mathop{\vrule width 5pt height 3 pt depth -2.6pt
                  \kern -6pt \intop}\nolimits_{\kern -3pt{#1}}}}
                  
\begin{document}

\thanks{The research was supported by the Academy of Finland, grant no.\ 135561.}

\begin{abstract}
Motivated by the results of Korry 
and Kinnunen and Saksman, 
we study the behaviour of the discrete fractional maximal operator on fractional Haj\l asz spaces, 
Haj\l asz--Besov and Haj\l asz--Triebel--Lizorkin spaces on metric measure spaces. 
We show that the discrete fractional maximal operator maps these spaces to the spaces of the same type with higher smoothness.
Our results extend and unify aforementioned results. We present our results in general setting, but they are new already in the Euclidean case.
\end{abstract}


\date{\today}  
\maketitle

\section{Introduction}
Maximal functions are standard tools in harmonic analysis. They are usually used to estimate absolute size, but recently there has been interest in studying their regularity properties, see 
\cite{AK}, \cite{AP}, \cite{B}, \cite{HM}, \cite{HO}, \cite{HKNT}, \cite{HLNT}, \cite{K}, \cite{KL}, \cite{KLi}, \cite{KT}, 
\cite{Ko2}, \cite{Ko3}, \cite{L}, \cite{L2}, \cite{T}.
A starting point was \cite{K}, where Kinnunen observed that the Hardy-Littlewood maximal operator
is bounded  on $W^{1,p}(\rn)$ for $1<p\le\infty$. In \cite{Ko2} and \cite{Ko3} Korry extended this result by showing that
the maximal operator preserves also fractional Sobolev spaces as well as Besov and Triebel--Lizorkin spaces. 
Another kind of extension was given in \cite{KS}, where Kinnunen and Saksman showed that the
fractional maximal operator $\M_\alpha$, defined by
\[
\M_\alpha u(x)=\sup_{r>0}\frac{r^\alpha}{|B(x,r)|}\int_{B(x,r)}|u(y)|\,dy, 
\]
is bounded from $W^{1,p}(\rn)$ to $W^{1,p^*}(\rn),$ where $p^*=np/(n-\alpha p)$, 
and from $L^p(\rn)$ to $\dot W^{1,q}(\rn)$, where $q=np/(n-(\alpha-1)p)$ and $\dot W^{1,q}(\rn)$ is the homogenous Sobolev space. 
These results indicate that $\M_\alpha$ has similar smoothing properties as the Riesz potential.

It is natural to ask whether these results can be seen as special cases of the behaviour of the fractional maximal operator on Besov and Triebel--Lizorkin spaces.
In this paper we show that this is indeed the case, and that all these results can be obtained
by the same rather simple method.
Instead of the standard fractional maximal operator, we consider its variant, the so-called discrete fractional maximal operator 
$\M_\alpha^*$. This allows us to present our results in a setting of doubling metric measure spaces. In this generality, the standard fractional maximal operator behaves quite badly. Indeed, one can construct spaces, where the fractional maximal function
of a Lipschitz function fails to be continuous, see \cite{B} and \cite{HLNT}.
Since $\M_\alpha^*$ and $\M_\alpha$ are comparable, for practical purposes it does not matter which one we choose.  
The discrete fractional maximal operator was introduced in \cite{KL} and further studied in \cite{KT}, \cite{AK} and \cite{HKNT}.

Among the many possible definitions of Besov and Triebel--Lizorkin spaces, the most suitable for our purposes is the one based on Haj\l asz type pointwise inequalities. This approach, introduced by
Koskela, Yang and Zhou in \cite{KYZ}, provides a new point of view to the classical Besov and Triebel--Lizorkin spaces. On the other hand, it allows these spaces to be defined in the setting of metric measure spaces.

By employing this definition, 
we can prove very general results using only simple ''telescoping'' arguments and Poincar\'e type inequalities. 
As special cases, we obtain versions of the results of Kinnunen and Saksman as well as those of Korry, see Remark \ref{remark: KS} and Theorems \ref{thm: TL} and \ref{thm: Besov}. 
We prove our results in doubling metric measure spaces but they are new even in Euclidean spaces. 
Our main results (Theorems \ref{thm: homog TL} and \ref{thm: homog Besov}) imply that if $\alpha\ge 0$ and $0<s+\alpha<1$, then $\M^*_\alpha$ is bounded from $\dot F^s_{p,q}(\rn)$ to 
$\dot F^{s+\alpha}_{p,q}(\rn)$ for $n/(n+s)<p,q<\infty$ and from $\dot B^s_{p,q}(\rn)$ to 
$\dot B^{s+\alpha}_{p,q}(\rn)$ for $n/(n+s)<p<\infty$, $0<q<\infty$, see Section \ref{section TL} for the definition of Triebel--Lizorkin and Besov spaces.

\section{Preliminaries and notation}
We assume that $X=(X, d,\mu)$ is a metric measure space
equipped with a metric $d$ and a Borel regular outer
measure $\mu$, which satisfies $0<\mu(U)<\infty$ whenever $U$ is nonempty, open and bounded.
We assume that the measure is doubling, that is, there exists a fixed constant
$c_d>0$, called the doubling constant, such that 
\begin{equation}\label{doubling measure}
\mu(B(x,2r))\le c_d\mu(B(x,r))
\end{equation}
for every ball $B(x,r)=\{y\in X:d(y,x)<r\}$. 

The doubling condition implies that
\begin{equation}\label{doubling dimension}
\frac{\mu(B(y,r))}{\mu(B(x,R))}\ge C\Big(\frac rR\Big)^Q 
\end{equation}
for every $0<r\le R$ and $y\in B(x,R)$ for some $C$ and $Q>1$ that only depend on $c_D$. 
In fact, we may take $Q=\log_2c_d$.

For the boundedness of the fractional maximal operator in $L^{p}$, we have to assume, in Theorems \ref{fracM bounded} and \ref{sobo norm}.(b), that the measure  $\mu$ satisfies the lower bound condition
\begin{equation}\label{measure lower bound}
\mu(B(x,r))\ge c_{l}r^{Q}
\end{equation} 
with some constant $c_{l}>0$ for all $x\in X$ and $r>0$.

Throughout the paper, $C$ will denote a positive constant whose value is not
necessarily the same at each occurrence.  

\subsection*{The fractional maximal function}
Let $\alpha\ge0$. 
The fractional maximal function of a locally integrable function $u$ is 
\begin{equation}\label{maximal function}
\M_\alpha u(x)=\sup_{r>0}\,r^{\alpha}\vint{B(x,r)}|u|\,d\mu,
\end{equation}
where $u_B=\vint{B}u\,d\mu=\frac1{\mu(B)}\int_Bu\,d\mu$ is the integral average of  $u$ over $B$.
For $\alpha=0$, we have the usual Hardy-Littlewood maximal function 
\[
\M u(x)=\sup_{r>0}\,\vint{B(x,r)}|u|\,d\mu. 
\]
The following Sobolev type inequality for the fractional maximal operator follows easily from the boundedness of the Hardy-Littlewood maximal operator in $L^{p}$, for the proof, see \cite{EKM}, \cite{GGKK} or \cite{HLNT}.

\begin{theorem}\label{fracM bounded}
Assume that the measure lower bound condition holds. 
If $p>1$ and $0<\alpha<Q/p$,
then there is a constant $C>0$, depending only on 
the doubling constant, constant in the measure lower bound, $p$ and $\alpha$, such that
 \[
 \|\M_{\alpha}u\|_{L^{p^*}(X)}
 \le C\|u\|_{L^p(X)},
 \]
for every $u\in L^{p}(X)$ with $p^*=Qp/(Q-\alpha p)$. 
\end{theorem}

\begin{remark} 
If $u$ is only locally integrable, then $\M_\alpha u$ may well be identically infinite. 
However, if $\M_\alpha u(x_0)<\infty$ for some $x_0\in X$, then $\M_\alpha u(x)< \infty$ for 
almost every $x$. This follows from the estimate
\[
r^{\alpha}\vint{B(x,r)}|u|\,d\mu
\le \frac{\mu(B(x_0,r+d(x,x_0)))}{\mu(B(x,r))}\M_\alpha u(x_0)
\]
combined with the doubling condition and the fact that
\[
\lim_{r\to 0}r^\alpha\vint{B(x,r)}|u|\,d\mu<\infty,
\]
whenever $x$ is a Lebesgue point of $u$.
\end{remark}

\subsection*{The discrete fractional maximal function}
We begin the construction of the discrete maximal function with a covering of the space.
Let $r>0$.
Since the measure is doubling, there are balls $B(x_i,r)$, $i=1,2,\dots$, such that 
\[
X=\bigcup_{i=1}^\infty B(x_i,r)
\quad\text{and}\quad
\sum_{i=1}^\infty \chi_{B(x_i,6r)}\le N<\infty,
\]
where $\chi_{B(x_i,6r)}$ is the characteristic function of the ball $B(x_i,6r)$.
This means that the dilated balls $B(x_i,6r)$, $i=1,2,\dots$, are of bounded overlap. 
The constant $N$ depends only on the doubling constant and, 
in particular, it is independent of $r$.

Then we construct a partition of unity subordinate to the covering 
$B(x_i,r)$, $i=1,2,\dots$, of $X$.
Indeed, there is a family of functions $\varphi_i$, $i=1,2,\dots$, such that
$0\le\varphi_i\le1$, 
$\varphi_i=0$ in $X\setminus B(x_i,6r)$,
$\varphi_i\ge\nu$ in $B(x_i,3r)$, 
$\varphi_i$ is Lipschitz with constant $L/r$ 
with $\nu$ and $L$ depending only on the doubling constant, and
\[
\sum_{i=1}^\infty\varphi_i(x)=1
\]
for every $x\in X$.

The discrete convolution of a locally integrable function $u$ at the scale $3r$ is
\[
u_{r}(x)=\sum_{i=1}^{\infty}\ph_{i}(x)u_{B(x_i,3r)}
\]
for every $x\in X$, and we write $u_{r}^{\alpha}=r^{\alpha}u_{r}$.

Let $r_{j}$, $j=1,2,\dots$ be an enumeration of the positive rationals and let balls 
$B(x_{i},r_{j})$, $i=1,2,\dots$ be a covering of $X$ as above.
The discrete fractional maximal function of $u$ in $X$ is
\[
\M^{*}_{\alpha}u(x)=\sup_{j}|u|_{r_j}^{\alpha}(x)
\] 
for every $x\in X$.
For $\alpha=0$, we obtain the Hardy-Littlewood type discrete maximal function $\M^{*}$ 
studied in \cite{KL},  \cite{KT} and \cite{AK}.
The discrete fractional maximal function is easily seen to be comparable to the standard fractional maximal function, see \cite{HKNT}.

\section{Fractional Haj\l asz spaces}\label{section M}
Let $u$ be a measurable function and let $s\ge 0$. A nonnegative measurable function $g$ is 
an $s$-Haj\l asz gradient of $u$ 
if there exists $E\subset X$ with $\mu(E)=0$ such that for all $x,y\in X\setminus E$,
\begin{equation}\label{eq: gradient}
|u(x)-u(y)|\le d(x,y)^s(g(x)+g(y)).
\end{equation}
The collection of all $s$-Haj\l asz gradients of $u$ is denoted by $\mathcal{D}^s(u)$.
A homogeneous Haj\l asz space $\dot{M}^{s,p}(X)$ consists of measurable functions $u$ such that
\[
\|u\|_{\dot M^{s,p}(X)}=\inf_{g\in\mathcal{D}^s(u)}\|g\|_{L^p(X)}
\]
is finite. The  Haj\l asz space $M^{s,p}(X)$ is $\dot M^{s,p}(X)\cap L^p(X)$ equipped with the norm
\[
\|u\|_{M^{s,p}(X)}=\|u\|_{L^p(X)}+\|u\|_{\dot M^{s,p}(X)}.
\]
The space $M^{1,p}(X)$, a counterpart of a Sobolev space in metric measure space, was introduced in \cite{H}, see also \cite{H2}. The fractional spaces $M^{s,p}(X)$ were introduced in \cite{Y} and studied for example in \cite{Hu} and \cite{HKT}. Notice that $M^{0,p}(X)=L^p(X)$. 

The pointwise definition of the Haj\l asz spaces implies the validity of Sobolev-Poincar\'e type inequalities without the assumption that the space admits any weak Poincar\'e inequality.


\begin{lemma}[\cite{GKZ}]\label{lemma:Sobolev-Poincare}
Let $s\in [0,\infty)$ and let $p\in (0,Q/s)$. There exists a constant $C$ such that for all measurable functions 
$u$ with $g\in \mathcal{D}^s(u)$, all $x\in X$ and $r>0$,
\begin{equation}\label{Sobolev-Poincare}
\begin{split}
\inf_{c\in\mathbb{R}}\bigg(\,\vint{B(x,r)}|u(y)-c|^{p^*(s)}\,d\mu(y)\bigg)^{1/p^*(s)}
\le Cr^s\bigg(\,\vint{B(x,2r)}g^{p}\,d\mu\bigg)^{1/p},
\end{split}
\end{equation}
where $p^*(s)=Qp/(Q-sp)$.
\end{lemma}

Moreover, if $p\ge Q/(Q+s)$ and $g\in \mathcal{D}^s(u)\cap L^p(X)$, then \eqref{Sobolev-Poincare} implies that $u$ is locally integrable and that 
\begin{equation}\label{Poincare}
\begin{split}
\vint{B(x,r)}|u-u_{B(x,r)}|\,d\mu
\le Cr^s \bigg(\,\vint{B(x,2r)}g^{p}\,d\mu\bigg)^{1/p}.
\end{split}
\end{equation}
For the case $s=1$, see \cite{H} and \cite{H2}.

In the next theorem, we use the following simple result.
 If $u_{i}$, $i\in\n$, are measurable functions with a common $s$-Haj\l asz gradient $g$ and 
 $u=\sup_{i}u_{i}$ is finite almost everywhere, then $g$ is an $s$-Haj\l asz gradient of $u$.

\begin{theorem}\label{sobo gradient}
Assume that $\M_\alpha^*u\not\equiv\infty$. 
Let $t\ge Q/(Q+s)$ and let $g$ be an $s$-Haj\l asz gradient of $u$. 
\begin{itemize}
\item[a)] If $0< s+\alpha\le 1$, then there exists
a constant $C>0$ such that
\[
\tilde g=C\left(\M g^t\right)^{1/t}
\]
is an $(s+\alpha)$-Haj\l asz gradient of 
$\M_\alpha^* u$.
\item[b)] If $s+\alpha > 1$, then there exists
a constant $C>0$ such that
\[
\tilde g=C\left(\M_{t(s+\alpha-1)} g^t\right)^{1/t}
\]
is a $1$-Haj\l asz gradient of  $\M_\alpha^* u$.
\end{itemize}
\end{theorem}

\begin{proof}
We begin by proving the claims for $u_{r}^{\alpha}$.
Let $r>0$, let $g\in\mathcal D^{s}(u)$ and let $x,y\in X$. 

Assume first that $r\ge d(x,y)$. Let $I_{xy}$ be a set of indices $i$ for which $x$ or $y$ belongs to 
$B(x_{i},6r)$. Then, for each $i\in I_{xy}$, $B(x_{i},3r)\subset B(x,10r)\subset B(x_{i},17r)$.
This together with the doubling condition, the properties of the functions $\ph_{i}$, 
the fact that there are bounded number of indices in $I_{xy}$ and 
Poincar\'e inequality \eqref{Poincare} implies that
 \begin{equation}\label{iso r}
\begin{aligned}
 |u_{r}^{\alpha}(x)-u_{r}^{\alpha}(y)|
 &\le r^{\alpha}\sum_{i\in I_{xy}}|\ph_{i}(x)-\ph_{i}(y)||u_{B(x_{i},3r)}-u_{B(x,10r)}|\\
 &\le C r^{\alpha-1}d(x,y)\sum_{i\in I_{xy}}\vint{B(x,10r)}|u-u_{B(x_{i},3r)}|\,d\mu\\
 &\le C r^{\alpha-1}d(x,y)\vint{B(x,10r)}|u-u_{B(x,10r)}|\,d\mu\\
&\le C r^{s+\alpha-1}d(x,y)\bigg(\,\vint{B(x,20r)}g^{t}\,d\mu\bigg)^{1/t}.
\end{aligned}
\end{equation}
If $0<s+\alpha\le1$, then by \eqref{iso r} and the assumption $r\ge d(x,y)$, we have that
\[
 |u_{r}^{\alpha}(x)-u_{r}^{\alpha}(y)|
 \le C d(x,y)^{s+\alpha}\big(\M g^{t}(x)\big)^{1/t}.
\]
If $s+\alpha>1$, then by \eqref{iso r},
\[
 |u_{r}^{\alpha}(x)-u_{r}^{\alpha}(y)|
 \le C d(x,y)\big(\M_{t(s+\alpha-1)} g^{t}(x)\big)^{1/t}.
\]
This shows that Haj\l asz gradient inequality \eqref{eq: gradient} with desired exponent holds when $r\ge d(x,y)$.

Assume then that $r<d(x,y)$. Let $R=d(x,y)$. Then $B(y,r)\subset B(x,2R)$ and 
\begin{equation}\label{sobo pieni r}
\begin{aligned}
  |u_{r}^{\alpha}(x)-u_{r}^{\alpha}(y)|
 \le r^{\alpha}\Big(&\sum_{i\in I_{x}}\ph_{i}(x)|u_{B(x_{i},3r)}-u_{B(x,9R)}|\\
                                + &\sum_{i\in I_{y}}\ph_{i}(y)|u_{B(x_{i},3r)}-u_{B(x,9R)}|\Big),
\end{aligned}
\end{equation}
where $I_{x}$ is a set of indices $i$ for which $x$ belongs to $B(x_{i},6r)$ and $I_{y}$ the corresponding set for $y$. 
Let $k\in\n$ be the smallest integer such that $2^{k}r\ge R$.

Assume first that $0<s+\alpha\le1$. If $i\in I_{x}$,
then 
\begin{equation}\label{pieni r}
\begin{aligned}
 |u_{B(x_{i},3r)}-u_{B(x,9R)}|
 &\le  |u_{B(x_{i},3r)}-u_{B(x,9r)}| + \sum_{i=1}^{k}|u_{B(x,2^{i}9r)}-u_{B(x,2^{i-1}9r)}|\\
& +|u_{B(x,2^{k}9r)}-u_{B(x,9R)}|.
\end{aligned}
\end{equation}
By the doubling condition and Poincar\'e inequality \eqref{Poincare}, we have
\begin{equation}\label{pieni r 1}
\begin{aligned}
r^{\alpha}|u_{B(x_{i},3r)}-u_{B(x,9r)}| 
&\le C r^{\alpha}\vint{B(x,9r)}|u-u_{B(x,9r)}|\,d\mu\\
&\le C r^{s+\alpha}\bigg(\,\vint{B(x,18r)}g^{t}\,d\mu\bigg)^{1/t}\\
&\le C R^{s+\alpha}\big(\M g^{t}(x)\big)^{1/t},
\end{aligned}
\end{equation}
and, by the doubling condition, Poincar\'e inequality \eqref{Poincare}, the fact that $r\le 2^{i}9r$ for all 
$i$, and  the selection of $k$, 
\begin{equation}\label{pieni r 2}
\begin{aligned}
r^{\alpha}\sum_{i=1}^{k}|u_{B(x,2^{i}9r)}-u_{B(x,2^{i-1}9r)}|
&\le Cr^{\alpha}\sum_{i=1}^{k}\vint{B(x,2^{i}9r)}|u-u_{B(x,2^{i}9r)}|\,d\mu\\
&\le C\sum_{i=1}^{k}(2^{i}9r)^{s+\alpha}\bigg(\,\vint{B(x,2^{i+1}9r)}g^{t}\,d\mu\bigg)^{1/t}\\
&\le CR^{s+\alpha}\big(\M g^{t}(x)\big)^{1/t}.
\end{aligned}
\end{equation}
Similarly we obtain that
\begin{equation}\label{pieni r 3}
\begin{aligned}
r^{\alpha}|u_{B(x,2^{k}9r)}-u_{B(x,9R)}|
&\le CR^{s+\alpha}\bigg(\,\vint{B(x,36R)}g^{t}\,d\mu\bigg)^{1/t}\\
&\le CR^{s+\alpha}\big(\M g^{t}(x)\big)^{1/t}.
\end{aligned}
\end{equation}
If $i\in I_{y}$, we use balls $B(y,2^{i}9r)$ instead of balls $B(x,2^{i}9r)$ in \eqref{pieni r}. 
Estimates corresponding \eqref{pieni r 1} and \eqref{pieni r 2} are as above ($x$ replaced by $y$) and, 
corresponding to \eqref{pieni r 3},
\begin{equation}\label{pieni r 3 y}
\begin{aligned}
r^{\alpha}|u_{B(y,2^{k}9r)}-u_{B(x,9R)}|
&\le CR^{s+\alpha}\bigg(\,\vint{B(x,38R)}g^{t}\,d\mu\bigg)^{1/t}\\
&\le CR^{s+\alpha}\big(\M g^{t}(x)\big)^{1/t}.
\end{aligned}
\end{equation}
Now, by \eqref{sobo pieni r}-\eqref{pieni r 3 y} and the fact $R=d(x,y)$, we have
\[
 r^{\alpha} |u_{B(x_{i},3r)}-u_{B(x,R)}|
 \le C d(x,y)^{s+\alpha}\Big(\big(\M g^{t}(x)\big)^{1/t}+\M g^{t}(y)\big)^{1/t}\Big).
\]
If $s+\alpha>1$, then similar estimates as above show that if $i\in I_{x}\cup I_{y}$, then 
\begin{equation}\label{iso s alpha}
 r^{\alpha} |u_{B(x_{i},3r)}-u_{B(x,R)}|
\le C d(x,y)\Big(\big(\M_{t(s+\alpha-1)} g^{t}(x)\big)^{1/t}
+\big(\M_{t(s+\alpha-1)} g^{t}(y)\big)^{1/t}\Big).
\end{equation}
These estimates together with \eqref{sobo pieni r} and the fact that there are bounded number of indices in $I_{x}$ and $I_{y}$ imply that Haj\l asz gradient inequality \eqref{eq: gradient} with desired exponent holds when 
$r<d(x,y)$.

The claim for $u_{r}^{\alpha}$ follows from the estimates above and for $\M_\alpha^* u$ from the discussion before the theorem. 
\end{proof}

\begin{theorem}\label{sobo norm}
Let  $Q/(Q+s)<p<\infty$. 
\begin{itemize}
\item[a)] If $0< s+\alpha\le 1$, there exists a constant $C>0$, such that 
\[
\|\M_\alpha^* u\|_{\dot M^{s+\alpha,p}(X)}\le C\|u\|_{\dot M^{s,p}(X)}
\]
for all  $u\in\dot{M}^{s,p}(X)$ with $\M^*_\alpha u\not\equiv\infty$.
\item[b)] If $1<s+\alpha\le 1+Q/p$ and the measure lower bound condition holds, there exists
a constant $C>0$ such that 
\[
\|\M_\alpha^* u\|_{\dot M^{1,q}(X)}\le C\|u\|_{\dot M^{s,p}(X)},
\]
where $q=Qp/(Q-(s+\alpha-1)p)$, for all  $u\in\dot{M}^{s,p}(X)$ with $\M^*_\alpha u\not\equiv\infty$.
\end{itemize}
\end{theorem}

\begin{proof}
a) Let $Q/(Q+s)\le t<p$. By Theorem \ref{sobo gradient}, the function $C(\M g^t)^{1/t}$ is an 
$(s+\alpha)$-gradient of $\M_\alpha^* u$. 
Since $g\in L^{p}(X)$, the claim follows from the boundedness of the Hardy--Littlewood maximal operator in $L^{q}(X)$ for $q>1$.
\medskip

b) Let $Q/(Q+s)\le t<p$. By Theorem \ref{sobo gradient}, the function $(\M_{t(s+\alpha-1)} g^t)^{1/t}$
is a $1$-gradient of  $\M_\alpha^* u$. 
Since $g\in L^{p}(X)$, the claim follows from Theorem \ref{fracM bounded}.
\end{proof}

\begin{remark}\label{remark: KS}
In the cases $s=0$ and $s=1$ of Theorem \ref{sobo norm}.b), we obtain counterparts of the results
of Kinnunen and Saksman.
\end{remark}

\begin{remark}
As a special case of Theorems \ref{sobo gradient} and \ref{sobo norm} we obtain boundedness results for the discrete maximal operator $\M^{*}$ in  $\dot M^{s,p}(X)$.  
If $0< s\le 1$, then $\tilde g=C(\M g^t)^{1/t}$ is an $s$-Haj\l asz gradient of  $\M^* u$ for all 
$t\ge Q/(Q+s)$ and 
\[
\|\M^* u\|_{\dot M^{s,p}(X)}\le C\|u\|_{\dot M^{s,p}(X)}
\]
for all  $u\in\dot{M}^{s,p}(X)$, $p>Q/(Q+s)$.

If $1<s\le 1+Q/p$, then 
$\tilde g=C(\M_{t(s-1)} g^t)^{1/t}$
is a $1$-Haj\l asz gradient of  $\M^* u$ for all $t\ge Q/(Q+s)$ and
\[
\|\M^* u\|_{\dot M^{1,q}(X)}\le C\|u\|_{\dot M^{s,p}(X)},
\]
where $q=Qp/(Q-(s-1)p)$, for all  $u\in\dot{M}^{s,p}(X)$.

Moreover, when $s=1$, we obtain boundedness results for the discrete maximal operator $\M^{*}$ in (homogeneous) 
Haj\l asz spaces $M^{1,p}(X)$, proved earlier for $M^{1,p}(X)$ in \cite{KL} and \cite{KT}.  
\end{remark}

\section{Haj\l asz--Besov and Haj\l asz--Triebel--Lizorkin spaces}\label{section TL}
Let $u$ be a measurable function and let $s\in(0,\infty)$. 
Following \cite{KYZ}, we say that a sequence of nonnegative measurable functions 
$(g_k)_{k\in\mathbb{Z}}$ is a fractional $s$-Haj\l asz gradient of $u$ if there exists $E\subset X$ 
with $\mu(E)=0$ such that
\[
|u(x)-u(y)|\le d(x,y)^s(g_k(x)+g_k(y))
\]
 for all $k\in\mathbb Z$ and all $x,y\in X\setminus E$ satisfying 
$2^{-k-1}\le d(x,y)<2^{-k}$.
The collection of all fractional $s$-Haj\l asz gradients of $u$ is denoted by $\mathbb{D}^s(u)$.

For $p\in(0,\infty)$, $q\in(0,\infty]$ and a sequence $(f_k)_{k\in\mathbb{Z}}$ of measurable 
functions, we write
\[
\|(f_k)_{k\in\mathbb{Z}}\|_{L^p(X,\,l^q)}
=\left\|\|\left(f_k\right)_{k\in\mathbb{Z}}\|_{l^q}\right\|_{L^p(X)}
\]
and
\[
\|(f_k)_{k\in\mathbb{Z}}\|_{l^q(L^p(X))}=
\big\|\left(\|f_k\|_{L^p(X)}\right)_{k\in\mathbb{Z}}\big\|_{l^q},
\]
where $\|(f_k)\|_{l^{q}}=(\sum_{k\in\mathbb{Z}}|f_{k}|^{q})^{1/q}$ if $0<q<\infty$ and 
$\|(f_k)\|_{l^{\infty}}=\sup_{k\in\mathbb{Z}}|f_{k}|$.

The homogeneous Haj\l asz--Triebel--Lizorkin space $\dot M_{p,q}^s(X)$ consists of measurable functions $u$ such that
\[
\|u\|_{\dot M_{p,q}^s(X)}=\inf_{(g_k)\in\mathbb{D}^s(u)}\|(g_k)\|_{L^p(X,\,l^q)}
\]
is finite. 
The Haj\l asz--Triebel--Lizorkin space $M_{p,q}^s(X)$ is $\dot M_{p,q}^s(X)\cap L^{p}(X)$
equipped with the norm
\[
\|u\|_{M_{p,q}^s(X)}=\|u\|_{L^p(X)}+\|u\|_{\dot M_{p,q}^s(X)}.
\]

The homogeneous Haj\l asz--Besov space $\dot N_{p,q}^s(X)$ consists of measurable functions $u$ such that
\[
\|u\|_{\dot N_{p,q}^s(X)}=\inf_{(g_k)\in\mathbb{D}^s(u)}\|(g_k)\|_{l^q(L^p(X))}
\]
is finite and the Haj\l asz--Besov space $N_{p,q}^s(X)$ is $\dot N_{p,q}^s(X)\cap L^p(X)$ 
equipped with the norm
\[
\|u\|_{N_{p,q}^s(X)}=\|u\|_{L^p(X)}+\|u\|_{\dot N_{p,q}^s(X)}.
\]

Notice that $\dot M_{p,\infty}^s(X)$ is the homogeneous fractional Haj\l asz space $\dot M^{s,p}(X)$,  for the simple proof, see \cite[Prop.\ 2.1]{KYZ}.
The homogeneous Haj\l asz--Triebel--Lizorkin space $\dot M^s_{p,q}(\rn)$ coincides with the 
classical homogeneous Triebel-Lizor\-kin space $\dot F^s_{p,q}(\rn)$ 
for $s\in (0,1)$, $p\in(n/(n+s),\infty)$ and $q\in(n/(n+s),\infty]$. Similarly, $\dot N_{p,q}^s(\rn)$ coincides with the classical homogeneous Besov space $\dot B^s_{p,q}(\rn)$ for $s\in (0,1)$, $p\in(n/(n+s),\infty)$ and $q\in(0,\infty]$ by \cite[Thm 1.2]{KYZ}. For the definitions of $F^s_{p,q}(\rn)$ and $B^s_{p,q}(\rn)$, see \cite{Tr}.

If $X$ supports a (weak) $(1,p)$-Poincar\'e inequality with $p\in (1,\infty)$, then for all $q\in(0,\infty)$, the spaces $\dot M^1_{p,q}(X)$ and $\dot N^1_{p,q}(X)$ are trivial, that is, they contain
only constant functions, see \cite[Thm 4.1]{GKZ}.

\begin{lemma}[\cite{GKZ}]\label{lemma:Sobolev-Poincare 2}
Let $s\in (0,\infty)$ and $p\in (0,Q/s)$. Then for every $\eps,\eps'\in(0,s)$ with $\eps<\eps'$ there exists a constant $C>0$ such that for all measurable functions $u$ with $(g_j)\in \mathbb{D}^s(u)$, $x\in X$ and $k\in\mathbb{Z}$,
\begin{equation}\label{Sobolev-Poincare 2}
\begin{split}
&\inf_{c\in\mathbb{R}}\bigg(\,\vint{B(x,2^{-k})}|u(y)-c|^{p^*(\eps)}\,d\mu(y)\bigg)^{1/p^*(\eps)}\\
&\le C2^{-k\eps'}\sum_{j\ge k-2}2^{-j(s-\eps')}
\bigg(\,\vint{B(x,2^{-k+1})}g_j^{p}\,d\mu\bigg)^{1/p},
\end{split}
\end{equation}
where $p^*(\eps)=Qp/(Q-\eps p)$.
\end{lemma}
If $p\ge Q/(Q+\eps)$, then \eqref{Sobolev-Poincare 2} implies that
\begin{equation}\label{Poincare 2}
\begin{split}
\vint{B(x,2^{-k})}|u-u_{B(x,2^{-k})}|\,d\mu
\le C2^{-k\eps'}\sum_{j\ge k-2}2^{-j(s-\eps')}
\bigg(\,\vint{B(x,2^{-k+1})}g_j^{p}\,d\mu\bigg)^{1/p}.
\end{split}
\end{equation}

We are now ready to state and prove our main results. Theorem \ref{main} below gives a formula for
an $(s+\alpha)$-Haj\l asz gradient of $\M^*_\alpha$ in terms of an $s$-Haj\l asz gradient of $u$.
This easily implies the desired boundedness results for $\M^*_\alpha$ in homogeneous Haj\l asz--Besov and Haj\l asz--Triebel--Lizorkin spaces. For related results concerning
Riesz potentials in the metric setting, see \cite{Y2}

\begin{theorem}\label{main}
Assume that $\M^*_\alpha u\not\equiv\infty$ and that $(g_k)\in\mathbb{D}^s(u)$.
Let $0< s+\alpha<1$, $0<\delta<1-s-\alpha$, $0<\eps<\eps'<s$ and $t\ge Q/(Q+\eps)$.
Then there is a constant $C>0$, indepent of $u$ and $(g_{k})$, such that $(C\tilde g_k)$, where
\begin{equation}\label{eq:gradient}
\tilde g_k=\sum_{j=-\infty}^{k} 2^{(j-k)\delta}\left(\M g_j^t\right)^{1/t}  \ 
+ \ \sum_{j=k-7}^\infty 2^{(k-j)(s-\eps')}\left(\M g_j^t\right)^{1/t}, 
\end{equation}
is a fractional $(s+\alpha)$-Haj\l asz gradient of $\M_\alpha^* u$.
\end{theorem}

\begin{proof}
Let $k\in\mathbb{Z}$ and let $x,y\in X$ such that $2^{-k-1}\le d(x,y)<2^{-k}$.
We will show that
\[
|u_r^\alpha(x)-u_r^\alpha(y)|\le Cd(x,y)^{s+\alpha}(\tilde g_k(x)+\tilde g_k(y)),
\]
where $C$ is independent of $r$ and $k$.

Assume first that $d(x,y)>r$. Then
\[
\begin{split}
|u_r(x)-u_r(y)|&\le |u_r(x)-u_{B(x,2^{-k+4})}|+|u_r(y)-u_{B(x,2^{-k+4})}|\\
&\le\sum_{i\in I_{x}}\varphi_i(x)|u_{B(x_i,3r)}-u_{B(x,2^{-k+4})}|\\
&+\sum_{i\in I_{y}}\varphi_i(y)|u_{B(x_i,3r)}-u_{B(x,2^{-k+4})}|,
\end{split}
\]
where $I_{x}$ is a set of indices $i$ for which $x$ belongs to $B(x_{i},6r)$ and $I_{y}$ the corresponding set for $y$. 
Let $m\in\mathbb{Z}$ be such that $2^{-m-1}< 9r\le 2^{-m}$. Since $r<d(x,y)<2^{-k}$, it follows that
$m\ge k-4$. 
If $i\in I_{x}$,
we obtain
\[
\begin{split}
|u_{B(x_i,3r)}-u_{B(x,2^{-k+4})}|
&\le |u_{B(x_i,3r)}-u_{B(x,2^{-m})}|+\sum_{l=k-4}^{m-1}|u_{B(x,2^{-l})}-u_{B(x,2^{-l-1})}|\\
&\le C\sum_{l=k-4}^{m}\,\vint{B(x,2^{-l})}|u-u_{B(x,2^{-l})}|\,d\mu
\end{split}
\]
and hence Poincar\'e inequality \eqref{Poincare 2} implies that
\[
\begin{split}
|u_{B(x_i,3r)}-u_{B(x,2^{-k+4})}|
&\le C\sum_{l=k-4}^\infty 2^{-l\eps'}\sum_{j=l-2}^\infty 2^{-j(s-\eps')}
        \left(\M g_j^t(x)\right)^{1/t}\\
&= C\sum_{j=k-6}^\infty 2^{-j(s-\eps')}\left(\M g_j^t(x)\right)^{1/t}\sum_{l=k-4}^{j+2}  2^{-l\eps'}\\
&\le C2^{-k\eps'}\sum_{j=k-6}^\infty 2^{-j(s-\eps')}\left(\M g_j^t(x)\right)^{1/t}\\
&= C2^{-ks}\sum_{j=k-6}^\infty 2^{(k-j)(s-\eps')}\left(\M g_j^t(x)\right)^{1/t}\\
&\le C2^{-ks}\tilde g_k(x).
\end{split}
\]
Similarly, if $i\in I_{y}$, then 
\[
\begin{split}
|u_{B(x_i,3r)}-u_{B(x,2^{-k+4})}|
&\le |u_{B(x_i,3r)}-u_{B(y,2^{-m})}|+\sum_{l=k-4}^{m-1}|u_{B(y,2^{-l})}-u_{B(y,2^{-l-1})}|\\
&+|u_{B(y,2^{-k+5})}-u_{B(x,2^{-k+4})}|\\
&\le C\sum_{l=k-5}^{m}\,\vint{B(y,2^{-l})}|u-u_{B(y,2^{-l})}|\,d\mu,
\end{split}
\]
which implies that
\[
|u_{B(x_i,3r)}-u_{B(x,2^{-k+4})}|\le C2^{-ks}\tilde g_k(y).
\]
It follows that
\[
\begin{split}
|u_r^\alpha(x)-u_r^\alpha(y)|&\le Cr^\alpha 2^{-ks}(\tilde g_k(x)+\tilde g_k(y))\\
&\le Cd(x,y)^{s+\alpha}(\tilde g_k(x)+\tilde g_k(y)).
\end{split}
\]
Suppose then that $d(x,y)\le r$. 
Let $I_{xy}$ be a set of indices $i$ for which $x$ or $y$ belongs to 
$B(x_{i},6r)$. Let $l$ be such that $2^{-l-1}<10r\le 2^{-l}$. 
Using the doubling condition, the properties of the functions $\ph_{i}$, 
the fact that there are bounded number of indices in $I_{xy}$ and 
Poincar\'e inequality \eqref{Poincare 2}, we have that
\begin{equation}\label{tl iso r1}
\begin{aligned}
 |u_r^\alpha(x)-u_r^\alpha(y)|
&\le r^\alpha\sum_{i=1}^\infty|\varphi_i(x)-\varphi_i(y))||u_{B(x_i,3r)}-u_{B(x,2^{-l})}|\\
&\le Cd(x,y)r^{\alpha-1}2^{-l\eps'}\sum_{j=l-2}^\infty 2^{-j(s-\eps')}\left(\M g_j^t(x)\right)^{1/t}.
\end{aligned}
\end{equation}
Using the assumptions $0<\delta<1-\alpha-s$, $r\ge d(x,y)$ and $d(x,y)<2^{-k}$, we have that
\[
\begin{split}
 d(x,y)r^{\alpha-1}2^{-l\eps'}
&\le Cd(x,y)r^{s+\alpha+\delta-1} \ 2^{l(s-\eps'+\delta)}
\le Cd(x,y)^{s+\alpha+\delta} \ 2^{l(s-\eps'+\delta)}\\
&\le Cd(x,y)^{\alpha+s}2^{(l-k)\delta+l(s-\eps')}.
\end{split}
\]
This together with \eqref{tl iso r1} implies that
\[
\begin{split}
|u_r^\alpha(x)-u_r^\alpha(y)|
\le \ &Cd(x,y)^{s+\alpha}
\sum_{j=l-2}^\infty 2^{(l-k)\delta+(l-j)(s-\eps')}\left(\M g_j^t(x)\right)^{1/t}.
\end{split}
\]
By splitting the sum in two parts and using the estimates $l\le j+2$ and $l\le k$, we obtain
\[
\begin{split}
& \sum_{j=l-2}^\infty 2^{(l-k)\delta+(l-j)(s-\eps')}\left(\M g_j^t(x)\right)^{1/t}\\
=\ & \sum_{j=l-2}^{k-1} 2^{(l-k)\delta+(l-j)(s-\eps')}\left(\M g_j^t(x)\right)^{1/t}
+ \sum_{j=k}^\infty 2^{(l-k)\delta+(l-j)(s-\eps')}\left(\M g_j^t(x)\right)^{1/t}\\
\le \ &C\Big(\sum_{j=-\infty}^{k-1}2^{(j-k)\delta}\left(\M g_j^t(x)\right)^{1/t} \ + \
\sum_{j=k}^{\infty}2^{(k-j)(s-\eps')}\left(\M g_j^t(x)\right)^{1/t}\Big),
\end{split}
\]
which implies the claim for $u_{r}^{\alpha}$. 
The claim for $\M^*_{\alpha} u$ follows similarly as in the proof of Theorem \ref{sobo gradient}.
\end{proof}

\begin{theorem}\label{thm: homog TL} 
Let $0<s+\alpha<1$ and $Q/(Q+s)<p,q<\infty$. 
Then there exists a constant $C>0$ such that
\[
\|\M_\alpha^* u\|_{\dot M_{p,q}^{s+\alpha}(X)}
\le C\|u\|_{\dot M_{p,q}^{s}(X)}
\]
for all $u\in \dot M_{p,q}^{s}(X)$ with $\M^*_\alpha u\not\equiv\infty$.
\end{theorem}
\begin{proof}
Let 
$\delta=\frac12(1-(s+\alpha))$, 
$\eps=\frac12\max\{s,s+\frac{Q-Qr}r\}$, $\eps'=\frac12(\eps+s)$, 
where $r=\min\{p,q\}$, and let $t=Q/(Q+\eps)$. Then $0<\eps<\eps'<s$ and $Q/(Q+s)<t<\min\{p,q\}$.  
By Theorem \ref{main}, $(C\tilde g_k)$ defined by \eqref{eq:gradient}
is a fractional $(s+\alpha)$-Haj\l asz gradient of $\M_\alpha^* u$.

It suffices to show that $(\tilde g_k)\in L^p(X,l^q)$. 
We estimate the $L^p(X,l^q)$ norm of
\[
\Big(\sum_{j=-\infty}^{k}2^{(j-k)\delta}\left(\M g_j^t\right)^{1/t}\Big)_{k\in\mathbb Z},
\]
the other part can be estimated similarly. 
If $q\ge 1$, we have, by the H\"older inequality, that
\[
\begin{split}
\sum_{k\in\mathbb Z}\Big(\sum_{j=-\infty}^{k}2^{(j-k)\delta}\big(\M g_j^t\big)^{1/t}\Big)^q
&\le C\sum_{k\in\mathbb Z}\sum_{j=-\infty}^{k}2^{(j-k)\delta}\left(\M g_j^t\right)^{q/t}\\
&\le C\sum_{j\in\mathbb Z}\left(\M g_j^t\right)^{q/t}\sum_{k=j}^{\infty}2^{(j-k)\delta}\\
&\le C\sum_{j\in\mathbb Z}\left(\M g_j^t\right)^{q/t}.
\end{split}
\]
If $q<1$, we obtain the same estimate by using the elementary inequality 
$(\sum_{j} a_j)^q\le \sum_{j} a_j^q$ for $a_{j}\ge 0$. 

By the Fefferman--Stein vector valued maximal function theorem from \cite{FS} 
(for a metric space version, see for example \cite{S} or \cite{GLY}), we obtain now
the desired estimate
\[
\begin{split}
\Big\|\Big(\sum_{j=-\infty}^{k}2^{(j-k)\delta}\left(\M g_j^t\right)^{1/t}\Big)_
  {k\in\mathbb{Z}}\Big\|_{L^p(X,\,l^q)}
&\le C\big\|\left(\M g_k^t\right)_{k\in\mathbb{Z}}\big\|_{L^{p/t}(X,\,l^{q/t})}^{1/t}\\
&\le C\|(g_k^t)_{k\in\mathbb{Z}}\|_{L^{p/t}(X,\,l^{q/t})}^{1/t}\\
&=C\|(g_k)_{k\in\mathbb{Z}}\|_{L^{p}(X,\,l^{q})}.
\end{split}
\]
\end{proof}

\begin{theorem}\label{thm: homog Besov}
Let $0<s+\alpha<1$, $Q/(Q+s)<p<\infty$ and $0<q<\infty$. 
Then there exists a constant $C>0$ such that
\[
\|\M_\alpha^* u\|_{\dot N_{p,q}^{s+\alpha}(X)}\le C\|u\|_{\dot N_{p,q}^{s}(X)}
\]
for all $u\in \dot N_{p,q}^{s}(X)$ with $\M^*_\alpha u\not\equiv\infty$.
\end{theorem}

\begin{proof}
Let 
$\delta=\frac12(1-(s+\alpha))$, 
$\eps=\frac12\max\{s,s+\frac{Q-Qp}p\}$, $\eps'=\frac12(\eps+s)$, 
and let $t=Q/(Q+\eps)$. Then $0<\eps<\eps'<s$ and $Q/(Q+s)<t<p$.  
Then $(C\tilde g_k)$ defined by \eqref{eq:gradient}
is a fractional $(s+\alpha)$-Haj\l asz gradient of $\M_\alpha^* u$ by Theorem \ref{main}.

It suffices to show that $\|(\tilde g_k)\|_{l^q(L^p(X))}\le C\|(g_k)\|_{l^q(L^p(X))}$. 
By the Hardy--Littlewood maximal theorem,
\[
\begin{split}
\Big\|\sum_{j=-\infty}^{k}2^{(j-k)\delta}\left(\M g_j^t\right)^{1/t}\Big\|_{L^p(X)}
&\le \sum_{j=-\infty}^{k}2^{(j-k)\delta}\big\|\left(\M g_j^t\right)^{1/t}\big\|_{L^p(X)}\\
&\le \sum_{j=-\infty}^{k}2^{(j-k)\delta}\|g_j\|_{L^p(X)}.
\end{split}
\]
If $q\ge 1$, we have by the H\"older inequality,
\[
\begin{split}
\sum_{k\in\mathbb Z}\Big(\sum_{j=-\infty}^{k}2^{(j-k)\delta}\|g_j\|_{L^p(X)}\Big)^q
&\le C\sum_{k\in\mathbb Z}\sum_{j=-\infty}^{k}2^{(j-k)\delta}\|g_j\|_{L^p(X)}^q\\
&\le C\sum_{j\in\mathbb Z}\|g_j\|_{L^p(X)}^q\sum_{k=j}^{\infty}2^{(j-k)\delta}\\
&\le C\sum_{j\in\mathbb Z}\|g_j\|_{L^p(X)}^q.
\end{split}
\]
If $q<1$, we use the inequality 
$(\sum_{j} a_j)^q\le \sum_{j} a_j^q$ instead of the Hölder inequality.
The second part of $(\tilde g_k)$ can be estimated similarly. 
\end{proof}

Theorems \ref{thm: homog TL}, \ref{thm: homog Besov} and the Hardy--Littlewood maximal theorem
imply the following results for the discrete maximal operator.

\begin{theorem}\label{thm: TL} 
Let $0<s<1$. 
\begin{itemize}
\item[a)] If $Q/(Q+s)<p,q<\infty$, then there exist a constant $C>0$ such that
\[
\|\M^* u\|_{\dot M_{p,q}^{s}(X)}\le C\|u\|_{\dot M_{p,q}^{s}(X)},
\]
whenever $u\in \dot M_{p,q}^{s}(X)$ and $\M^* u\not\equiv\infty$.
\item[b)] If $1<p,q<\infty$, then there exist a constant $C>0$ such that
\[
\|\M^* u\|_{M_{p,q}^{s}(X)}\le C\|u\|_{M_{p,q}^{s}(X)},
\]
for all $u\in M_{p,q}^{s}(X)$.
\end{itemize}
\end{theorem}

\begin{theorem}\label{thm: Besov}
Let $0<s<1$.
\begin{itemize}
\item[a)] If $Q/(Q+s)<p<\infty$ and $0<q<\infty$, there exist a constant $C>0$ such that
\[
\|\M^* u\|_{\dot N_{p,q}^{s}(X)}\le C\|u\|_{\dot N_{p,q}^{s}(X)}
\]
for all $u\in \dot N_{p,q}^{s}(X)$ with $\M^* u\not\equiv\infty$.
\item[b)]  If $1<p<\infty$ and $0<q<\infty$, there exist a constant $C>0$ such that
\[
\|\M^* u\|_{N_{p,q}^{s}(X)}\le C\|u\|_{ N_{p,q}^{s}(X)}
\]
for all $u\in N_{p,q}^{s}(X)$.
\end{itemize}
\end{theorem}

\vspace{0.5cm}
\noindent
\small{\textsc{T.H.},}
\small{\textsc{Department of Mathematics},}
\small{\textsc{P.O. Box 11100},}
\small{\textsc{FI-00076 Aalto University},}
\small{\textsc{Finland}}\\
\footnotesize{\texttt{toni.heikkinen@aalto.fi}}

\vspace{0.3cm}
\noindent
\small{\textsc{H.T.},}
\small{\textsc{Department of Mathematics and Statistics},}
\small{\textsc{P.O. Box 35},}
\small{\textsc{FI-40014 University of Jyv\"askyl\"a},}
\small{\textsc{Finland}}\\
\footnotesize{\texttt{heli.m.tuominen@jyu.fi}}

\end{document}